\documentclass[reqno]{amsart}
\usepackage{enumerate,amsmath}
\usepackage{pb-diagram}
\usepackage{pstricks,pst-node,pst-coil,pst-plot}
\usepackage{graphicx}

\theoremstyle{plain}

\newtheorem{thm}{Theorem}[section]
\newtheorem{theorem}[thm]{Theorem}

\newtheorem{lemma}[thm]{Lemma}

\newtheorem{proposition}[thm]{Proposition}
\theoremstyle{remark}

\theoremstyle{definition}

\newtheorem{definition}[thm]{Definition}
\def\al{{\alpha}}

\def\de{{\delta}}

\def\om{{\omega}}
\def\Om{{\Omega}}
\def\la{{\lambda}}

\def\si{{\sigma}}
\def\Si{{\Sigma}}
\def\ga{{\gamma}}
\def\ep{{\varepsilon}}

\def\phi{{\varphi}}

\let\pa\partial
\let\na\nabla

\DeclareMathAlphabet{\doba}{U}{msb}{m}{n}

\gdef\mN{\doba{N}}

\gdef\mR{\doba{R}}
\gdef\mS{\doba{S}}

\def\cU{{\mathcal U}}

\let\vol\Vol
\def\Scal{{\mathop{\rm Scal}}}

\def\Ric{{\mathop{\rm Ric}}}

\let\spec\Spec

\let\ti\tilde

\def\eref#1{{\rm (\ref{#1})}}
\def\Id{{\mathop{\mathrm{Id}}}}

\let\<\langle
\let\>\rangle

\def\Riem{\mbox{\textit{Riem}}^n}
\def\Riemspin{\mbox{\textit{Riemspin}}^n}
\def\Laplace{\mbox{\textit{Laplace}}^n_k}
\def\Dirac{\mbox{\textit{Dirac}}^n_k}

\def\proofof#1{\par\medbreak\noindent {\bf Proof of \ignorespaces #1.\enspace}}

\long\def\ignorethis#1{}

\begin{document}
\title{The supremum of conformally covariant eigenvalues in a conformal class}
\author{Bernd Ammann, Pierre Jammes}

\begin{abstract}
Let $(M,g)$ be a compact Riemannian manifold of dimension $\geq 3$.
We show that there is a metrics $\tilde g$ conformal to $g$ and of volume $1$
such that the first positive eigenvalue the conformal Laplacian with 
respect to $\ti g$ is
arbitrarily large. A similar statement is proven for the first positive
eigenvalue of the Dirac operator on a spin manifold of dimension $\geq 2$.
\end{abstract}

\maketitle

\begin{center}
\today
\end{center}

\footnote{bernd.ammann at gmx.de, pierre.jammes at univ-avignon.fr}

\setcounter{tocdepth}{1}
\tableofcontents

\section{Introduction}

The goal of this article is to prove the following theorems.
\begin{theorem}\label{theo.main.dir}
Let $(M,g_0,\chi)$ be compact Riemannian spin manifold of dimension $n\geq 2$.
For any metric~$g$
in the conformal class $[g_0]$, we denote the first positive eigenvalue of
the Dirac operator on $(M,g,\chi)$ by
$\la_1^+(D_g)$.
Then
   $$\sup_{g\in [g_0]}\la_1^+(D_g)\vol(M,g)^{1/n}=\infty.$$
\end{theorem}

\begin{theorem}\label{theo.main.yam}
Let $(M,g_0,\chi)$ be compact Riemannian manifold of dimension $n\geq 3$.
For any metric~$g$
in the conformal class $[g_0]$, we denote the first positive eigenvalue of
the conformal Laplacian $L_g:=\Delta_g + \frac{n-2}{4(n-1)}\Scal_g$ (also called Yamabe operator)
on $(M,g,\chi)$ by $\la_1^+(L_g)$.
Then
   $$\sup_{g\in [g_0]}\la_1^+(L_g)\vol(M,g)^{2/n}=\infty.$$
\end{theorem}

The Dirac operator and the conformal Laplacian belong to a large family
of operators, definded in details in subsection~\ref{subsec.cov.conf}.
These operators are called conformally covariant elliptic operators
of order $k$ and of bidegree $((n-k)/2,(n+k)/2)$,
acting on manifolds $(M,g)$ of dimension $n>k$. 
In particular, our definition includes formal self-adjointness.

The above theorems can be generalized to the following:

\begin{theorem}\label{theo.main}
Let $P_g$ be a conformally covariant elliptic operator of order $k$,
of bidegree $((n-k)/2,(n+k)/2)$
acting on manifolds of dimension $n>k$. We also assume
that $P_g$ is invertible on $\mS^{n-1}\times \mR$ (see Definition~\ref{def.invsphere}).
Let $(M,g_0)$ be compact Riemannian manifold. In the case that
$P_g$ depends on the spin structure, we assume that $M$ is oriented and is
equipped with a spin structure.
For any metric~$g$ in the conformal class $[g_0]$, we denote the first
positive eigenvalue of $P_g$ by $\la_1^+(P_g)$.
Then
   $$\sup_{g\in [g_0]}\la_1^+(P_g)\vol(M,g)^{k/n}=\infty.$$
\end{theorem}

The interest in this result is motivated by three questions.
At first, we found that the infimum
$$\inf_{g\in [g_0]}\la_1^+(D_g)\vol(M,g)^{1/n}$$
has a rich geometrical structure \cite{ammann:habil}, \cite{ammann:p03},
\cite{ammann.humbert.morel:06}, \cite{ammann.humbert:06a}.
In particular it is strictly positive \cite{ammann:03}
and under some condition
preventing the blowup of spheres it is attained  \cite{ammann:habil}, 
\cite{ammann:p03}.

The second motivation comes from comparing this result to results about
other differential operators. Let us recall that for the Hodge Laplacian
$\Delta_p^g$ acting on $p$-forms, we have
$\sup_{g\in [g_0]}\la_1(\Delta_p^g)\vol(M,g)^{2/n}=+\infty$ for
$n\geq4$ and $2\leq p\leq n-2$ (\cite{colbois.elsoufi:06}). 
On the other hand,
for the standard Laplacian $\Delta^g$ acting on functions, we have
$\sup_{g\in [g_0]}\la_k(\Delta^g)\vol(M,g)^{2/n}<+\infty$ (the case $k=1$
is proven in \cite{elsoufi.ilias:86} and the general case in
\cite{korevaar:93}). See \cite{jammes:07} for a synthetic presentation of
this subject.

The essential idea in the proof is to construct metrics with longer and
longer cylindrical parts. We will call this an
\emph{asymptotically cylindrical blowup}.
Such metrics are also called \emph{Pinocchio metrics} in
\cite{ammann:diss,ammann.baer:00}. In \cite{ammann:diss,ammann.baer:00}
the behavior of Dirac eigenvalues
on such metrics has already been studied partially, but
the present article has much stronger results.
This provides the third motivation.

{\bf Acknowledgements}
We thank B. Colbois, M. Dahl, E. Humbert and O. Hijazi 
for many related discussions.
We thank R. Gover for some helpful comments on conformally covariant operators,
and for several references.
The first author wants to thank cordially the Einstein institute at
Potsdam-Golm for its hospitality which enabled to write the article.

\section{Preliminaries}

\subsection{Notations}
In this article $B_y(r)$ denotes the ball of radius $r$ around $y$, $S_y(r)=\pa B_y(r)$
its boundary. The standard sphere $S_0(1)\subset \mR^n$ in $\mR^n$ is also denoted by $\mS^{n-1}$, its 
volume is $\om_{n-1}$. For the volume element of $(M,g)$ we use the notation $dv^g$. 

For sections $u$ of a vector bundle $V\to M$ over a Riemannian manifold $(M,g)$ 
the Sobolev norms $L^2$ and $H^s$, $s\in \mN$, are defined as
\begin{eqnarray*}
\|u\|_{L^2(M,g)}^2&:=&\int_M |u|^2\,dv^g\\
\|u\|_{H^s(M,g)}^2&:=&\|u\|_{L^2(M,g)}^2 + \|\na u\|_{L^2(M,g)}^2 +\ldots
    + \|\na^s u\|_{L^2(M,g)}^2.
\end{eqnarray*}
The vector bundle $V$ will be suppressed in the notation.
If $M$ and $g$ are clear from the context, we write just $L^2$ $H^s$.
The completions of the compactly supported sections of $V$
with respect to these norms are called $L^2(M,g)$ and $H^s(M,g)$.

\subsection{Removal of singularities}

In the proof we will use the following removal of singularity lemma.

\begin{lemma}[Removal of singularities lemma]\label{lem.sing.rem}
Let $\Omega$ be a bounded open subset of $\mR^n$ containing $0$.
Let $P$ be an elliptic differential operator of order $k$ on $\Omega$, 
$f\in C^\infty(\Om)$, and let $u\in C^\infty(\Om\setminus\{0\})$ be a solution
of
\begin{equation}\label{eq.P}
Pu=f
\end{equation}
on $\Om\setminus\{0\}$
with
\begin{equation}\label{cond.rem}
\lim_{\ep\to 0}\int_{B_0(2\ep)-B_0(\ep)} |u|r^{-k}=0 \mbox{ \ and \ }
\lim_{\ep\to 0}\int_{B_0(\ep)} |u| =0
\end{equation}
where $r$ is the distance to $0$.
Then $u$ is a (strong) solution of~\eref{eq.P}
on $\Om$.
The same result holds for sections of vector bundles over relatively compact open subset of Riemannian manifolds.
\end{lemma}

\begin{proof}
We show that $u$ is a weak solution of \eref{eq.P}, and then it follows
from standard regularity theory, that it is also a strong solution.
This means that we have to show that for any given compactly supported smooth
test function $\psi:\Om\to \mR$ we have
  $$\int_\Om u P^* \psi = \int_\Om f \psi.$$

Let $\eta:\Omega\to [0,1]$ be a test function that is identically $1$ on $B_0(\ep)$,
has support in $B_0(2\ep)$, and with
$|\na^m \eta |\leq C_m/\ep^m$. It follows that
  $$\sup |P^* (\eta \psi)|\leq C(P,\psi)\ep^{-k},$$
on $B_0(2\ep)\setminus B_0(\ep)$ and  $\sup |P^* (\eta \psi)|\leq C(P,\psi)$
on $B_0(\ep)$
and hence
\begin{equation}
\begin{split}
  \left| \int_\Om u P^*(\eta\psi)\right| &\leq C \ep^{-k} \int_{B_0(2\ep)
\setminus B_0(\ep)} |u| + C  \int_{B_0(\ep)} |u|\\
& \leq C \int_{B_0(2\ep)\setminus B_0(\ep)} |u|r^{-k}+ C  \int_{B_0(\ep)} |u|\to 0.
\end{split}
\end{equation}
We conclude
\begin{equation}
\begin{split}
  \int_\Om u P^*\psi &= \int_\Om u P^*(\eta \psi) +  \int_\Om u P^*((1-\eta) \psi)\\
          &= \underbrace{\int_\Om u P^*(\eta \psi)}_{\to 0}
    +  \underbrace{\int_\Om (Pu) (1-\eta) \psi}_{\to \int_\Om f\psi}
\end{split}
\end{equation}
for $\ep\to 0$. Hence the lemma follows.
\end{proof}

Condition \eref{cond.rem} is obviously satisfied if $\int_\Om |u|r^{-k}<\infty$.
It is also satisfied if
\begin{equation}\label{cond.rem.two}
  \int_\Om |u|^2r^{-k}<\infty\mbox{ and }k\leq n,
\end{equation}
as in this case
$$\left(\int_{B_0(2\ep)\setminus B_0(\ep)} |u|r^{-k}\right)^2 \leq \int_\Om
|u|^2 r^{-k}\underbrace{\int_{B_0(2\ep)\setminus B_0(\ep)}r^{-k}}_{\leq C}.$$
\subsection{Conformally covariant elliptic operators}\label{subsec.cov.conf}

In this subsection we present a class of certain conformally covariant
elliptic operators.
Many important geometric operators are  in this class, in particular the conformal Laplacian,
the Paneitz operator, the Dirac operator, see also  \cite{fegan:76,branson:98,gover.peterson:03} for more examples.

Such an operator is not just one single differential operator, but a procedure
how to associate to an $n$-dimensional Riemannian manifold $(M,g)$
(potentially with some additional structure)
a differential operator $P_g$ of order $k$ acting on a vector bundle.
The important fact is that if $g_2=f^2 g_1$, then one claims
\begin{equation}\label{eq.conf.trafo}
P_{g_2}= f^{-\frac{n+k}2} P_{g_1} f^{\frac{n-k}2}.
\end{equation}
One also expresses this by saying that $P$ has bidegree $((n-k)/2,(n+k)/2)$.

The sense of this equation is apparent if $P_g$ is an operator from
$C^\infty(M)$ to $C^\infty(M)$. If $P_g$ acts on a vector bundle or
if some additional structure (as e.g. spin structure) is used for
defining it, then a rigorous and
careful definition needs more attention.
The language of categories provides
a good formal framework \cite{maclane:98}.
The concept of conformally covariant elliptic
operators is already used by many authors, but we do not know of a reference
where a formal definition is carried out that fits to our context.
(See \cite{kolar.michor.slovak:93} for a similar categorial
approach that includes some of the operators presented here.)
Often an intuitive definition is used.
The intuitive definition is obviously sufficient if one deals with operators acting on functions,
such as the conformal Laplacian or the Paneitz operator.
However to properly state
Theorem~\ref{theo.main} we need the following definition.

Let $\Riem$ (resp. $\Riemspin$)  be the
category $n$-dimensional
Riemannian manifolds (resp. $n$-dimensional Riemannian
manifolds with orientation
and spin structure). Morphisms from $(M_1,g_1)$ to $(M_2,g_2)$
are conformal embeddings $(M_1,g_1)\hookrightarrow (M_2,g_2)$
(resp. conformal embeddings preserving orientation and spin structure).

Let $\Laplace$ (resp. $\Dirac$) be the category
whose objects are $\{(M,g),V_g,P_g\}$ where $(M,g)$ in an object of
$\Riem$ (resp. $\Riemspin$), where $V_g$ is a vector bundle with
a scalar product on the fibers, where $P_g:\Gamma(V_g)\to \Gamma(V_g)$
is an elliptic formally selfadjoint differential operator of order $k$.

A morphism $(\iota,\kappa)$
from $\{(M_1,g_1), V_{g_1},P_{g_1}\}$ to $\{(M_2,g_2), V_{g_2},P_{g_2}\}$
consists of a conformal embedding $\iota:(M_1,g_1)\hookrightarrow (M_2,g_2)$
(preserving orientation and spin structure in the case of $\Dirac$)
together with a fiber isomorphism $\kappa:\iota^*V_{g_2}\to V_{g_1}$
perserving
fiberwise length, such that $P_{g_1}$ and $P_{g_2}$
satisfy the conformal covariance
property \eref{eq.conf.trafo}.
For stating this property precisely, let $f>0$ be defined  by
$\iota^*g_2= f^2 g_1$, and let
$\kappa_*:\Gamma(V_{g_2})\to \Gamma(V_{g_1})$,
$\kappa_*(\phi)=\kappa\circ \phi\circ
\iota$. Then the conformal covariance property is
\begin{equation}\label{eq.conf.trafo.rig}
\kappa_*P_{g_2}= f^{-\frac{n+k}2} P_{g_1} f^{\frac{n-k}2}
\kappa_*.
\end{equation}

In the following the maps $\kappa$ and $\iota$ will often be evident
from the context and then will be omitted. The transformation
formula~\eref{eq.conf.trafo.rig} then simplifies to~\eref{eq.conf.trafo}.

\begin{definition}
A \emph{conformally covariant elliptic operator of order $k$ and
of bidegree $((n-k)/2,(n+k)/2)$} is
a contravariant functor from $\Riem$ (resp. $\Riemspin$) to $\Laplace$
(resp. $\Dirac$), mapping $(M,g)$ to $(M,g,V_g,P_g)$ in such a way that the
coefficients are continuous in the $C^k$-topology of metrics (see below).
To shorten notation, we just write $P_g$ or $P$ for this functor.
\end{definition}

It remains to explain the $C^k$-continuity of the coefficients.

For Riemannian metrics $g$, $g_1$, $g_2$ defined on a
compact set $K\subset M$ we set
  $$d^g_{C^k(K)}(g_1,g_2):=\max_{t=0,\ldots, k} \|(\na_g)^t(g_1-g_2)\|_{C^0(K)}.$$
For a fixed background metric $g$, the relation
$d^g_{C^k(K)}(\,\cdot\,,\,\cdot\,)$
defines a distance function on the space of metrics on $K$.
The topology induced by $d^g$ is independent of
this background metric and it is called the
\emph{$C^k$-topology of metrics on $K$}.



\begin{definition}
We say that the coefficients of $P$ are
\emph{continuous in the $C^k$-topology of metrics}
if for any metric $g$ on a manifold $M$, and for any compact subset $K\subset M$
there is a neighborhood $\cU$ of $g|_K$ in the
$C^k$-topology of metrics on $K$, such that for all metrics $\ti g$, $\ti g|_K\in\cU$, there is
an isomorphism of vector bundles
$\hat\kappa:V_g|_K\to V_{\ti g}|_K$ over the identity of $K$ with induced map
$\hat\kappa_*:\Gamma(V_g|_K)\to\Gamma(V_{\ti g}|_K)$ with the property that
the coefficients of the differential operator
 $$P_g-(\hat\kappa_*)^{-1}P_{\ti g}\hat\kappa_*$$
depend continuously on $\ti g$ (with repsect to the $C^k$-topology of metrics).
\end{definition}

\subsection{Invertibility on $\mS^{n-1}\times \mR$}

Let $P$ be a conformally covariant elliptic operator of order $k$
and of bidegree $((n-k)/2,(n+k)/2)$. For $(M,g)=\mS^{n-1}\times \mR$, the operator
$P_g$ is a self-adjoint operator $H^k\subset L^2\to L^2$ (see Lemma~\ref{b-lemma} 
and the comments thereafter). 

\begin{definition}\label{def.invsphere}
We say that \emph{$P$ is invertible on
$\mS^{n-1}\times \mR$} if $P_g$ is an invertible operator
$H^k\to L^2$ where~$g$ is the standard product metric
on $\mS^{n-1}\times \mR$. In order words
there is a constant $\si>0$ such that
the spectrum of $P_g:\Gamma_{H^k}(V_g)\to\Gamma_{L^2}(V_g)$
is contained in $(-\infty,-\sigma]\cup [\sigma,\infty)$ for any $g\in U$.
In the following, the largest such $\si$ will be called $\si_P$.
\end{definition}

We conjecture that any conformally covariant elliptic operator of order $k$ and
of bidegree $((n-k)/2,(n+k)/2)$ with $k<n$ is invertible
on $\mS^{n-1}\times \mR$.

\subsection{Examples}\ \\
{\it Example 1:} The Conformal Laplacian\\
Let
 $$L_g:= \Delta_g + \frac{n-2}{4(n-1)}\Scal_g,$$
be the conformal Laplacian. It acts on functions on a Riemannian manifold $(M,g)$, i.e. $V_g$ is the trivial real line bundle
$\underline{\mR}$.
Let $\iota:(M_1,g_1)\hookrightarrow (M_2,g_2)$ be a conformal
embedding. Then we can choose $\kappa:= \Id: \iota^*V_{g_2}\to V_{g_1}$
and formula \eref{eq.conf.trafo.rig} holds for $k=2$
(see e.g.\cite[Section 1.J]{besse:87}).
All coefficients of $L_g$ depend continuously on $g$ in the $C^2$-topology.
Hence $L$ is a conformally covariant elliptic operator of order $2$
and of bidegree $((n-2)/2,(n+2)/2)$.

The scalar curvature of $\mS^{n-1}\times \mR$ is $(n-1)(n-2)$.
Hence the spectrum of $L_g$ on $\mS^{n-1}\times \mR$ of $L_g$ coincides with the essential spectrum of $L_g$
and is $[\si_{L},\infty)$ with $\si_L:=(n-2)^2/4$. 
Hence $L$ is invertible on $S^{n-1}\times \mR$ if (and only if) $n>2$.

{\it Example 2:} The Paneitz operator

﻿Let $(M,g)$ be a smooth, compact Riemannian manifold of dimension $n\ge5$.
The Paneitz operator $P_g$ is given by
$$P_gu=(\Delta_g)^2 u-{\rm div}_g (A_g\,du)+\frac{n-4}{2}Q_g u$$
where
  $$A_g:= \frac{(n-2)^2+4}{2(n-1)(n-2)}\Scal_gg-\frac{4}{n-2}{\rm \Ric}_g,$$
  $$Q_g=\frac{1}{2(n-1)}\Delta_g\Scal_g+ \frac{n^3-4n^2+16n-16}{8(n-1)^2(n-2)^2}\Scal_g^2-\frac{2}{(n-2)^2}|\Ric_g|^2.$$
This operator was defined by Paneitz in the case $n=4$, and it was generalized
by Branson in \cite{branson:87} to arbitrary dimensions $\geq 4$. We also refer
to Theorem 1.21 of the overview article \cite{branson:85}. The explicit
formula presented above can be found e.g.\ in \cite{hebey.robert:01}.
The coefficients of $P_g$ depend continuous on $g$ in the $C^4$-topology

As in the previous example we can choose for $\kappa$ the identity, and
then the Paneitz operator $P_g$ is a conformally
covariant elliptic operator of order $4$ and of bidegree $((n-4)/2,(n+4)/2)$.

On $\mS^{n-1}\times\mR$ one calculates
  $$A_g:= {(n-4)n\over 2} \Id + 4 \pi_\mR>0$$
where $\pi_\mR$ is projection to vectors parallel to $\mR$.
  $$Q_g:= {(n-4)n^2\over 8}.$$
We conclude
   $$\si_P=Q= {(n-4)n^2\over 8}$$
and $P$ is invertible on $\mS^{n-1}\times \mR$
if (and only if) $n>4$.

{\it Examples 3:} The Dirac operator.

Let $\ti g= f^2 g$. Let $\Sigma^gM$ resp.\ $\Sigma^{\ti g}M$ be the spinor
bundle of $(M,g)$ resp. $(M,\ti g)$. Then there is a fiberwise isomorphism
$\beta_{\ti g}^g:\Sigma^g M\to \Sigma^{\ti g}M$, preserving the norm such that
  $$D_{\ti g}\circ \beta_{\ti g}^g(\phi)= f^{-\frac{n+1}2}
    \beta_{\ti g}^g\circ D_{g}\left(f^{\frac{n-1}2}\phi\right),$$
see  \cite{hitchin:74,baum:81,hijazi:01} for details.
Furthermore, the cocycle conditions
$$\beta^g_{\ti g} \circ\beta^{\ti g}_g= \Id\qquad\mbox{and}\qquad \beta^{\hat g}_g\circ \beta^{\ti g}_{\hat g}\circ \beta^g_{\ti g}=\Id$$
hold for conformal metrics $g$, $\ti g$ and $\hat g$.
We will hence use the map $\beta^g_{\ti g}$ to identify $\Si^gM$ with
$\Si^{\ti g}M$. Hence we simply get
\begin{equation}\label{eq.conf.trafo.dir}
D_{\ti g}\phi = f^{-\frac{n+1}2}
    \circ D_{g} \left( f^{\frac{n-1}2}\phi\right).
\end{equation}
All coefficients of $D_g$ depend continuously on $g$ in the $C^1$-topology.
Hence $D$ is a conformally covariant elliptic operator of order $1$
and of bidegree $((n-1)/2,(n+1)/2)$.

The Dirac operator on $\mS^{n-1}\times \mR$
can be decomposed as $D_{\rm vert}+ D_{\rm hor}$, where the first part 
is the sum over the derivations (and Clifford multiplication) 
along $\mS^{n-1}$ and where $D_{\rm hor}=\pa_t \cdot\na_{\pa_t}$, where
$\pa_t \cdot$ is Clifford multiplication with $\pa_t$, $t\in\mR$.
$D_{\rm vert}$ and $D_{\rm hor}$ anticommute. The spectrum of $D_{\rm vert}$
is just the spectrum of the Dirac operator on $\mS^{n-1}$, and hence we 
see with  \cite{baer:96a}
  $${\rm spec} D_{\rm vert}= \{\pm \left(\frac{n-1}{2}+k\right)\,|\,k\in \mN_0\}.$$
The operator $(D_{\rm hor})^2$ is the ordinary Laplacian on $\mR$ 
and hence has spectrum $[0,\infty)$. Together this implies that the 
spectrum of the Dirac operator on $\mS^{n-1}\times \mR$ 
is $(-\infty,-\si_D]\cup [\si_D,\infty)$ with $\si_D=\frac{n-1}{2}$.

Hence $D$ is invertible on $S^{n-1}\times \mR$ if (and only if) $n>1$.

In the case $n=2$ these statements are only correct if the circle
$\mS^{n-1}=\mS^1$
carries the spin structure induced from the ball. In our article all 
circles $\mS^1$ carry this bounding spin structure due to the geometry 
of the asymptotically cylindrical blowups. 

{\it Example 4:} The Rarita-Schwinger operator and many other Fegan type
operators are conformally covariant elliptic
operators of order $1$ and of bidegree
$((n-1)/2,(n+1)/2)$. See \cite{fegan:76} and in the work of T.~Branson for
more information.

{\it Example 5:} Assume that $(M,g)$ is a Riemannian spin manifold that
carries a vector bundle
$W\to M$ with metric and metric connection. Then there is a natural
first order operator $\Gamma(\Si M\otimes W)\to \Gamma(\Si M\otimes W)$, the
\emph{Dirac operator twisted by $W$}. This operator has similar properties
as conformally covariant elliptic operators of order $1$ and of bidegree
$((n-1)/2,(n+1)/2)$. The methods of our article can be easily adapted in order
to show that Theorem~\ref{theo.main} is also true for this
twisted Dirac operator. However, twisted Dirac operators are not
``conformally covariant elliptic operators'' in the above sense. They could
have been included in this class by replacing the category $\Riemspin$
by a category of Riemannian spin manifolds with twisting bundles.
In order not to overload the formalism we chose not to present these larger categories.

The same discussion applies to the ${\rm spin^c}$-Dirac operator of a
${\rm spin^c}$-manifold.

\section{Asymptotically cylindrical blowups}

\subsection{Convention}
{\it From now on we suppose that $P_g$ is a conformally covariant elliptic
operator of order $k$, of bidegree $((n-k)/2,(n+k)/2)$,
acting on manifolds of dimension $n$ and invertible on $\mS^{n-1}\times \mR$.}

\subsection{Definition of the metrics}
Let $g_0$ be a Riemannian metric on a compact manifold $M$.
We can suppose that the injectivity radius in a fixed point $y\in M$ is
larger than $1$. The geodesic distance from $y$ to $x$ is denoted by $d(x,y)$.

We choose a smooth
function $F_\infty:M\setminus\{y\}\to [1,\infty)$ such
such that $F_\infty(x)= 1$ if $d(x,y)\geq 1$, $F_\infty(x)\leq 2$ if
$d(x,y)\geq 1/2$ and
such that $F_\infty(x)=d(x,y)^{-1}$ if $d(x,y)\in (0,1/2]$.
Then for $L\geq 1$  we define  $F_L$ to be a smooth positive
function on $M$, depending only on $d(x,y)$, such that
$F_L(x)= F_\infty(x)$ if $d(x,y)\geq e^{-L}$ and
$F_L(x)\leq d(x,y)^{-1}=F_\infty(x)$ if $d(x,y)\leq e^{-L}$.

For any $L\geq 1$ or $L=\infty$ set  $g_L:= F_L^2 g_0$.
The metric $g_\infty$ is a
complete metric on $M_\infty$.

The family of metrics $(g_L)$ is called an \emph{asymptotically cylindrical
blowup}, in the literature it is denoted as a family of
\emph{Pinocchio metrics} \cite{ammann.baer:00},
see also Figure~\ref{fig.pinocchio}.

\begin{figure}[htb]
\begin{center}
\includegraphics{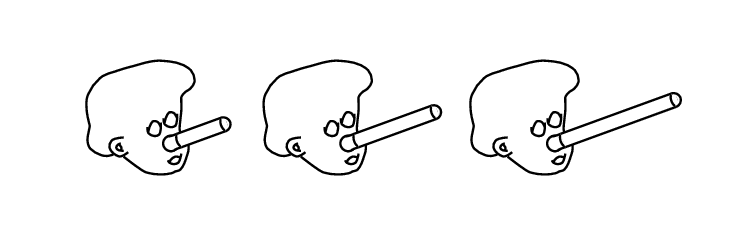}
\end{center}
\caption{Asymptotically cylindrical metrics $g_L$ (alias Pinocchio metrics)
with growing nose length $L$.}\label{fig.pinocchio}
\end{figure}

\subsection{Eigenvalues and basic properties on $(M,g_L)$}
For the $P$-operator associated to $(M,g_L)$, $L\in \{0\}\cup [1,\infty)$
(or more exactly
its selfadjoint extension) we simply write  $P_L$ instead of $P_{g_L}$.
As $M$ is compact the spectrum of $P_L$ is discrete.

We will denote the spectrum of $P_L$ in the following way
  $$\ldots \leq \la_1^-(P_{g_L})< 0= 0 \ldots=0<
     \la_1^+(P_{g_L}) \leq \la_2^+(P_{g_L})\leq \ldots,$$
where each eigenvalue appears with the multiplicity of the multiplicity of the
eigenspace. The zeros might appear on this list or not, depending
on whether $P_{g_L}$ is invertible or not.
The spectrum might be entirely positive
(for example the conformal Laplacian $Y_g$ on the sphere)
in which case $\la_1^-(P_{g_L})$ is not defined. Similarly,
$\la_1^+(P_{g_L})$ is not defined if the spectrum of $(P_{g_L})$ is negative.

\subsection{The asymptotic analysis of $(M_\infty, g_\infty)$}
The asymptotic analysis of non-compact manifolds as
$(M_\infty, g_\infty)$ is more complicated than
in the compact case. Nevertheless  $(M_\infty, g_\infty)$ is an
asymptotically cylindrical manifolf for there exists nowadays an
extensive literature. We will need only very few of these properties that will
be summarized in this subsection. Proofs will only be sketched.

Different approaches can be used for the proof.
The following lemma shows that $(M_\infty, g_\infty)$
carries a $b$-structure in the sense of Melrose.
``Manifolds with $b$-structures'' form a subclass of
``Manifolds with a Lie structure at infinity'', also called ``Lie  manifolds''
\cite{ammann.lauter.nistor:04}, \cite{ammann.lauter.nistor:07},
\cite{ammann.ionescu.nistor:06}.
We choose to use Melrose's $b$-calculus \cite{melrose:93} in this article as this
calculus is more widely known, but similar statements hold in the larger category
of ``Manifolds with a Lie structure at infinity''.

\begin{lemma}\label{b-lemma}
The manifold $(M_\infty,g_\infty)$ is an exact $b$-metric in the sense of
\cite[Def. 2.8]{melrose:93}.
\end{lemma}
\begin{proof} Let $S_yM$ be the unit tangent bundle at $y$.
For any $X\in S_yM$ let $\ga_X$ be the geodesic with $\dot\ga(0)=X$.
Then for small $\ep>0$ the map
  $\Phi:S_yM\times [0,\ep)\to M,\qquad (X,t)\mapsto\ga_X(t)$ is smooth
and a diffeomorphism from $S_yM\times (0,\ep)$ to $B_y(\ep)\setminus\{0\}$.
We define $\bar M:=M_\infty\cup :S_yM\times [0,\ep)/\sim$ where
$\sim$ indicates that $(X,t)\in S_yM\times (0,\ep)$ is glued together with
$\Phi(X,t)$. $\bar M$ is a manifold with boundary $S_yM$ and interior
$M_\infty$. Expressing the metric in normal coordinates, one sees that
$d(x,y)^2g$ extends to an exact $b$-metric on $\bar M$.
\end{proof}

From this observation may properties already follow with standard arguments 
similarly as in the compact case.
The operator $P_{g_\infty}$ has a
self-adjoint extension,
denoted by $P_\infty$.
The essential spectrum of $P_\infty$ coincides with
the essential spectrum of the $P$-operator on the standard cylinder
$\mS^{n-1}\times \mR$ which is contained in
$(-\infty,-\si_P]\cup [\si_P,\infty)$.
Hence the spectrum of $P_\infty$ in the interval $(-\si_P,\si_P)$ is
discrete as well. Eigenvalues of $P_\infty$ in this interval will be called
small eigenvalues of $P_\infty$.
Similarly we use the notation $\la_j^{\pm}(P_\infty)$
for the small eigenvalues of $P_\infty$.

\begin{proposition}\label{prop.interpol}
Let $P$ be a conformally covariant elliptic operator. Then on $(M_\infty,g_\infty)$ we have 
  $$\|(\na^\infty)^s u\|_{L^2(g_\infty)}\leq C (\|u\|_{L^2(g_\infty)}+ \|P_\infty u\|_{L^2(g_\infty)})$$
for all $s\in \{0,1,\ldots,k\}$.
\end{proposition}
\begin{proof}
Choose a  $\lambda \in\mR$ that is not in the spectrum of $P$.
The continuity of the coefficients of $P$ together with the fact that
$P_{g_0}$ extends to $M$ implies that $P_{g_\infty}$ is an operator
compatible with the $b$-structure. Hence we we can apply
\cite[Proposition]{melrose:93} for $Q:=P-\la$. We see that $P-\la$ is an
isomorphism from $H^k$ to $L^2$.
Hence a constant $C>0$ exists with
  $$C\|(P-\la)u\|_{L^2}\geq \| u\|_{H^k}.$$
using the triangle inequality we get
  $$\| u\|_{H^k}\leq  C \la \|u\|_{L^2(g_\infty)}+ \|P_\infty u\|_{L^2(g_\infty)}$$
which is equivalent to the statement.
\end{proof}

\subsection{The kernel}
Having recalled these previously known facts we will now
study the kernel of the conformally covariant operators.

If $g$ and $\ti g=f^2$ are conformal metrics on a compact manifold $M$, then 
  $$\phi \mapsto  f{-\frac{n-k}2} \phi$$
obviously defines an isomorphism from $\ker P_g$ to $\ker P_{\ti g}$.
It is less obvious that a similar statement holds if we compare  $g_0$ and $g_\infty$ defined before: 

\begin{proposition}
The map
\begin{eqnarray*}
\ker P_0&\to & \ker P_\infty\\
  \phi_0& \mapsto & \phi_\infty= F_\infty^{-\frac{n-k}2} \phi_0
\end{eqnarray*}
is an isomorphism of vector spaces.
\end{proposition}

\begin{proof}Suppose $\phi_0\in\ker P_0$.
Using standard regularity results
it is clear that $\sup |\phi_0|<\infty$. Then
\begin{equation}
\begin{split}
\int_{M_\infty}|\phi_\infty|^2\,dv^{g_\infty} & \leq
\int_{M\setminus B_y(1/2)} |\phi_\infty|^2\,dv^{g_\infty}
+ \sup |\phi_0|^2 \int_{B_y(1/2)} F_\infty^{-(n-k)}\,dv^{g_\infty}\\
&\leq 2^k \int_{M\setminus B_y(1/2)} |\phi_0|^2\,dv^{g_0}+   \sup |\phi_0|^2
\om_{n-1}\int_0^{1/2} \frac{r^{n-1}}{r^k}\, dr<\infty.
\end{split}
\end{equation}
Furthermore, formula~\eref{eq.conf.trafo}
implies $P_\infty\phi_\infty=0$.
Hence the map is well-defined. In order to show that it is an isomorphism
we show that the obvious inverse
$\phi_\infty\mapsto \phi_0:=F_\infty^{\frac{n-k}2}\phi_\infty$
is well defined. To see
this we start with an $L^2$-section in the kernel of $P_\infty$.

We calculate
  $$\int_M F_\infty^k |\phi_0|^2\,dv^{g_0}=
   \int_{M_\infty}|\phi_\infty|^2\,dv^{g_\infty}.$$
Using again~\eref{eq.conf.trafo} we see that this section satisfies
$P_0\phi_0$ on $M\setminus\{y\}$.
Hence condition~\eref{cond.rem.two} is satisfied, and
together with the removal of singularity lemma (Lemma~\ref{lem.sing.rem})
one obtains that the inverse map is well-defined.
The proposition follows.
\end{proof}

\section{Proof of the main theorem}

\subsection{Stronger version of the main theorem}

We will now show the following theorem.
\begin{theorem}\label{theo.conv}
Let $P$ be a  conformally covariant elliptic
operator of order $k$, of bidegree $((n-k)/2,(n+k)/2)$,
on manifolds of dimension $n>k$. We assume that $P$ is invertible on $\mS^{n-1}\times\mR$.

If $\liminf_{L\to \infty} |\la_j^\pm(P_L)|< \si_P$, then
  $$\la_j^\pm(P_L) \to \la_j^\pm(P_\infty)\in \left(-\si_P, \si_P\right)\qquad\mbox{for }L\to \infty.$$
\end{theorem}
In the case $\spec(P_{g_0})\subset(0,\infty)$
the theorem only makes a statement about $\la^+_j$, and conversely
in the case that $\spec(P_{g_0})\subset(-\infty,0)$ it only makes a statement
about $\la^-_j$.

Obviously this theorem implies Theorem~\ref{theo.main}.

\subsection{The supremum part of the proof of Theorem~\ref{theo.conv}}\label{subsec.sup}

At first we prove that
\begin{equation}\label{limsup}
\limsup_{L\to \infty} (\la_j^+ (P_{L}))\leq \la_j^+(P_\infty).
\end{equation}
Let $\phi_1,\ldots,\phi_j$ be sequence of $L^2$-orthonormal
eigenvectors of $P_\infty$ to eigenvalues
$\la_1^+(P_\infty),\ldots,\la_j^+(P_\infty)\in [-\bar\la,\bar\la]$,
$\bar\la<\si_P$.
We choose a cut-off function $\chi:M\to [0,1]$ with $\chi(x)=1$
for $-\log(d(x,y))\leq T$, $\chi(y)=0$ for $-\log (d(x,y))\geq 2T$, and
$|(\na^\infty)^s\chi|_{g_\infty}\leq C_s/T^s$ for all $s\in \{0,\ldots,k\}$.

Let $\phi$ be a linear combination of the eigenvectors $\phi_1,\ldots,\phi_j$.
From Proposition~\ref{prop.interpol}
we see that
 $$\|(\na^\infty)^s\phi\|_{L^2(M_\infty,g_\infty)}\leq C \|\phi\|_{L^2(M_\infty,g_\infty)}$$
where $C$ only depends on $(M_\infty,g_\infty)$.
Hence for sufficiently large $T$
 $$\|P_\infty(\chi \phi)-\chi P_\infty \phi\|_{L^2(M_\infty,g_\infty)}
   \leq kC/T \|\phi\|_{L^2(M_\infty,g_\infty)}
   \leq 2kC/T \|\chi\phi\|_{L^2(M_\infty,g_\infty)}$$
for sufficiently large $T$ as  $\|\chi \phi\|_{L^2(M_\infty,g_\infty)}\to
\| \phi\|_{L^2(M_\infty,g_\infty)}$ for $T\to \infty$.
The section $\chi \phi$ can be interpreted as a section on $(M,g_L)$ if $L>2T$,
and on the support of $\chi\phi$ we have $g_L=g_\infty$
and $P_\infty(\chi \phi) = P_L(\chi \phi)$.
Hence standard Rayleigh quotient arguments imply that if $P_\infty$ has $m$
eigenvalues (counted with mulitplicity) in the intervall $[a,b]$ then
$P_L$ has $m$ eigenvalues in the intervall $[a-2kC/T,b+2kC/T]$.
Taking the limit $T \to \infty$ we obtain  \eref{limsup}.

By exchanging some obvious signs we obtain similarly
\begin{equation}\label{limsup.neg}
\limsup_{L\to \infty} (-\la_j^- (P_{L}))\leq -\la_j^-(P_\infty).
\end{equation}

\subsection{The infimum part of the proof of Theorem~\ref{theo.conv}}\label{subsec.inf}
We now prove
\begin{equation}\label{liminf}
\liminf_{L\to \infty} (\pm \la_j^\pm (P_L))\geq\pm \la_j^\pm (P_\infty).
\end{equation}

We assume that we have a sequence $L_i\to \infty$, and that for each $i$
we have a system of orthogonal eigenvectors
$\phi_{i,1}$, \ldots, $\phi_{i,m}$
of $P_{L_i}$, i.e.\
$P_{L_i}\phi_{i,\ell} =\la_{i,\ell} \phi_{i,\ell}$
for $\ell\in \{1,\ldots,m\}$.
Furthermore we suppose that
$\la_{i,\ell}\to \bar\la_\ell\in(-\si_P,\si_P)$ for $\ell\in\{1,\ldots,m\}$.

Then
 $$\psi_{i,\ell}:=\left(\frac{F_{L_i}}{F_\infty}\right)^{\frac{n-k}2}\phi_{i,\ell}$$
satisfies
  $$P_\infty\psi_{i,\ell} = h_{i,\ell} \psi_{i,\ell}\qquad
    \mbox{with}\qquad h_{i,\ell}:=\left(\frac{F_{L_i}}{F_\infty}\right)^k\la_{i,\ell}.$$
Furthermore
 $$\|\psi_{i,\ell}\|_{L^2(M_\infty,g_\infty)}^2
   =\int_M \left(\frac{F_{L_i}}{F_\infty}\right)^{-k} |\phi_{i,\ell}|^2\,dv^{g_{L_i}}
   \leq \sup_M  |\phi_{i,\ell}|^2 \int_M \left(\frac{F_{L_i}}{F_\infty}\right)^{-k}\,dv^{g_{L_i}}$$
Because of  $\int_M \left(\frac{F_{L_i}}{F_\infty}\right)^{-k}\,dv^{g_L}\leq C\int r^{n-1-k}\,dr<\infty$
(for $n>k$) the norm $\|\psi_{i,\ell}\|_{L^2(M_\infty,g_\infty)}$ is finite as well,
and we can renormalize such that
  $$\|\psi_{i,\ell}\|_{L^2(M_\infty,g_\infty)}=1.$$

\begin{lemma}
For any $\delta>0$ and any $\ell\in\{0,\ldots,m\}$ the sequence
  $$\Big(\|\psi_{i,\ell}\|_{C^{k+1}(M\setminus B_y(\delta),g_\infty)}\Bigr)_i$$
is bounded.
\end{lemma}
\proofof{the lemma}
After removing finitely many $i$, we can assume that $\la_i\leq 2\bar\la$ and $e^{-L_i}<\delta/2$.
Hence $F_L=F_\infty$ and $h_i=\la_i$ on $M\setminus B_y(\delta/2)$.
Because of
  $$\int_{M\setminus B_y(\delta/2)} |(P_\infty)^s\psi_i|^2\,dv^{g_\infty}
    \leq (2\bar\la)^{2s} \int_{M\setminus B_y(\delta/2)} |\psi_i|^2\,dv^{g_\infty}\leq   (2\bar\la)^{2s}$$
we obtain boundedness of $\psi_i$ in the Sobolev space
$H^{sk}(M\setminus B_y(3\delta/4),g_\infty)$, and hence, for sufficiently
large~$s$ boudnedness in $C^{k+1}(M\setminus B_y(\delta),g_\infty)$.
The lemma is proved.
\qed

Hence after passing to a subsequence $\psi_{i,\ell}$ converges in
$C^{k,\al}(M\setminus B_y(\delta),g_\infty)$ to a solution $\bar\psi_\ell$ of
  $$P_\infty\bar\psi_\ell= \bar\la_\ell\bar\psi_\ell.$$
By taking a diagonal sequence, one can obtain convergence in
$C^{k,\al}_{\rm loc}(M_\infty)$ of $\psi_{i,\ell}$ to $\bar\psi_\ell$.
It remains to prove
that $\bar\psi_1$,\ldots,$\bar\psi_m$ are linearly independent,
in particular that any $\bar\psi_\ell\neq 0$. For this we use the following
lemma.

\begin{lemma}
For any $\ep>0$ there is $\delta_0$ and $i_0$ such that
  $$\Bigl\|\psi_{i,\ell}\Big\|_{L^2(B_y(\delta_0),g_\infty)}
    \leq \ep \Bigl\| \psi_{i,\ell}\Bigr\|_{L^2(M,g_\infty)}$$
for all $i\geq i_0$ and all $\ell\in\{0,\ldots,m\}$.
In particular,
 $$\Bigl\|\psi_{i,\ell}\Big\|_{L^2(M\setminus B_y(\delta_0),g_\infty)}
    \geq (1-\ep)\Bigl\| \psi_{i,\ell}\Bigr\|_{L^2(M,g_\infty)}.$$

\end{lemma}

\proofof{the lemma}
Because of Proposition~\ref{prop.interpol} and
  $$\|P_\infty\psi_{i,\ell}\|_{L^2(M_\infty,g_\infty)}\leq |\bar\la_\ell|\,\|\psi_{i,\ell}\|_{L^2(M_\infty,g_\infty)}= |\bar\la_\ell|$$
we get
  $$\|(\na^\infty)^s \psi_{i,\ell}\|_{L^2(M_\infty,g_\infty)}\leq C$$
for all $s\in\{0,\ldots,k\}$. Let $\chi$ be a cut-off dunction as in
Subsection~\ref{subsec.sup} with $T=-\log \delta$.
Hence
\begin{equation}\label{ineq.comm}
  \| P_\infty \Bigl((1-\chi) \psi_{i,\ell}\Bigr)- (1-\chi) P_\infty( \psi_{i,\ell})\|_{L^2(M_\infty,g_\infty)}\leq \frac{C}{T}= \frac{C}{-\log\de}.
\end{equation}
On the other hand $(B_y(\delta)\setminus\{y\},g_\infty)$
converges for $\delta\to 0$
to $\mS^{n-1}\times(0,\infty)$ in the $C^\infty$-topology.
Hence there is a function $\tau(\delta)$ converging to $0$ such that
\begin{equation}\label{ineq.pinf}
  \| P_\infty \Bigl((1-\chi) \psi_{i,\ell}\Bigr)\|_{L^2(M_\infty,g_\infty)}\geq
  (\si_p-\tau(\delta))\| (1-\chi) \psi_{i,\ell}\|_{L^2(M_\infty,g_\infty)}.
\end{equation}
Using the obvious relation
 $$\|(1-\chi) P_\infty( \psi_{i,\ell})\|_{L^2(M_\infty,g_\infty)}
    \leq |\la_{i,\ell}|\,  \|(1-\chi)\psi_{i,\ell}\|_{L^2(M_\infty,g_\infty)}$$
we obtain with \eref{ineq.comm} and \eref{ineq.pinf}
  $$ \|\psi_{i,\ell}\|_{L^2(B_y(\delta^2),g_\infty)}\leq
  \|(1-\chi)\psi_{i,\ell}\|_{L^2(M_\infty,g_\infty)}\leq \frac{C}{|\log \de|(\si_P-\tau(\delta)-|\la_{i,\ell}|)}.$$
The right hand side is smaller than $\ep$ for $i$ sufficiently large
and $\delta$ sufficiently small. The main statement of the lemma
then follows for $\de_0:=\de^2$.
The Minkowski inequality yields.
 $$\|\psi_{i,\ell}\|_{L^2(M\setminus B_y(\delta^2),g_\infty)}
  \geq  1-\|\psi_{i,\ell}\|_{L^2(B_y(\delta^2),g_\infty)}\geq 1-\ep.$$
\qed

The convergence in $C^1(M\setminus B_y(\delta_0))$ implies strong
convergence in $L^2(M\setminus B_y(\delta_0),g_\infty)$ of $\psi_{i,\ell}$ to
$\bar\psi_\ell$. Hence
  $$\|\bar\psi_\ell\|_{L^2(M\setminus B_y(\delta_0),g_\infty)}\geq 1-\ep,$$
and thus $\|\bar\psi_\ell\|_{L^2(M_\infty,g_\infty)}=1$.
The orthogonality of these sections is provided by the following lemma, and
the inequality~\eref{liminf} then follows immediatly.

\begin{lemma}The sections $\bar\psi_1,\ldots,\bar\psi_m$ are orthogonal.
\end{lemma}
\proofof{the lemma}
The sections $\phi_{i,1},\ldots,\phi_{i,\ell}$ are orthogonal.
For any fixed $\delta_0$ (given by the previous lemma),
it follows for sufficiently large $i$ that
\begin{equation}
\begin{split}
\Bigl|\int_{M\setminus B_y(\de_0)} \<\psi_{i,\ell},\psi_{i,\ti\ell}\>\, dv^{g_\infty}\Bigr|
  & = \Bigl |\int_{M\setminus B_y(\de_0)} \<\phi_{i,\ell},\phi_{i,\ti\ell}\>\, dv^{g_{L_i}}\Bigr|\\
  & = \Bigl|\int_{B_y(\de_0)} \<\phi_{i,\ell},\phi_{i,\ti\ell}\>\, dv^{g_{L_i}}\Bigr|\\
  & = \Bigl|\int_{B_y(\de_0)}  \underbrace{\left(\frac{F_{L_i}}{F_\infty}\right)^k}_{\leq 1}
      \<\psi_{i,\ell},\psi_{i,\ti\ell}\>\, dv^{g_\infty}\Bigr|\\
  & \leq \ep^2
\end{split}
\end{equation}
Because of strong $L^2$ convergence on $M\setminus B_y(\de_0)$ this implies
\begin{equation}
\Bigl|\int_{M\setminus B_y(\de_0)} \<\bar\psi_\ell,\bar\psi_{\ti\ell}\>\, dv^{g_\infty}\Bigr|\leq \ep^2
\end{equation}
for $\ti\ell\neq \ell$, 
and hence in the limit $\ep\to 0$ (and $\de_0\to 0$) we get the orthogonality of
 $\bar\psi_1,\ldots,\bar\psi_m$.
\qed


\providecommand{\bysame}{\leavevmode\hbox to3em{\hrulefill}\thinspace}
\providecommand{\MR}{\relax\ifhmode\unskip\space\fi MR }
\providecommand{\MRhref}[2]{%
  \href{http://www.ams.org/mathscinet-getitem?mr=#1}{#2}
}
\providecommand{\href}[2]{#2}

\vspace{1cm}
Authors' address:
\nopagebreak
\vspace{5mm}\\
\parskip0ex
\vtop{
\hsize=9cm\noindent
\obeylines
Bernd Ammann
Institut \'Elie Cartan BP 239
Universit\'e de Nancy 1
54506 Vandoeuvre-l\`es -Nancy Cedex
France
\vspace{0.5cm}
Pierre Jammes
Laboratoire d'analyse non lin\'eaire et g\'eom\'etrie
Universit\'e d'Avignon
33 rue Louis Pasteur
84000 Avignon
France

\vspace{0.5cm}
E-Mail:
{\tt bernd.ammann at gmx.net and pierre.jammes at univ-avignon.fr}         
}

\end{document}